\documentclass[reqno,b5paper]{amsart}
\usepackage{amsmath}
\usepackage{amssymb}
\usepackage{amsthm}
\usepackage{enumerate}
\usepackage[mathscr]{eucal}

\newtheorem{theorem}{Theorem}[section]

\begin{document}

\title[Embeddings in the Fell and Wijsman topologies]{Embeddings in the Fell and Wijsman topologies}
\author[\v Lubica Hol\'a ]{\v Lubica Hol\'a}

\newcommand{\acr}{\newline\indent}

\address{\llap{*}Academy of Sciences, Institute of Mathematics \acr \v
Stef\'anikova 49,
81473 Bratislava, Slovakia
\acr Slovakia}

\email{hola@mat.savba.sk}

\thanks{}

\subjclass[2010]{Primary 54B20; Secondary 54B05.}
\keywords{embedding, hyperspace, Fell topology, Wijsman topology,  $\omega_1$.
The author  would like to thank to grant APVV-0269-11 and  Vega 2/0018/13.}

\begin{abstract}

It is shown that if a $T_2$ topological space $X$ contains a closed uncountable discrete subspace, then  the spaces $(\omega_1 + 1)^{\omega}$ and  $(\omega_1 + 1)^{\omega_1}$ embed  into $(CL(X),\tau_F)$, the hyperspace of nonempty closed subsets of $X$ equipped  with the Fell topology. If $(X,d)$ is a non-separable perfect topological space, then $(\omega_1 + 1)^{\omega}$ and $(\omega_1 + 1)^{\omega_1}$ embed into $(CL(X),\tau_{w(d)})$, the hyperspace of nonempty closed subsets of $X$ equipped with the Wijsman topology, giving a partial answer to the Question 3.4 in [CJ].

\end{abstract}

\maketitle

\section{Introduction}
\bigskip
\bigskip
\bigskip

Throughout this paper, let $2^X$ ($CL(X)$) denote the family of all (nonempty) closed subsets of a given $T_2$ topological space. For $M \in CL(X)$, put

\bigskip
\centerline{$M^- = \{A \in CL(X): A \cap M \ne \emptyset\}$, $M^+ = \{A \in CL(X): A \subseteq M\}$}

\bigskip
and denote $M^c = X \setminus M$. The Vietoris topology [Mi] $\tau_V$ on $CL(X)$ has as subbase elements of the form $U^-$, $V^+$, where $U, V$ are open in $X$.

\bigskip
The Fell topology [Fe] $\tau_F$ on $CL(X)$ has as a subbase the collection

\bigskip
\centerline{$\{U^-: U$ open in $X\} \cup \{(K^c)^+: K$ compact in $X\}$.}

\bigskip
It si known [Be] that $(CL(X),\tau_F)$ is Hausdorff (regular, Tychonoff, respectively) iff $X$ is locally compact.
The normality problem of the Fell hyperspace was settled by Hol\'a, Levi and Pelant in [HLP], where they showed that $(CL(X),\tau_F)$ is normal if and only if $X$ is locally compact and Lindel\"{o}f.

\bigskip

For a metric space $(X,d)$, let $d(x,A) =$ inf$\{d(x,a): a \in A\}$ denote the distance between a point $x \in X$ and a nonempty subset $A$ of $(X,d)$.

A net $\{A_\alpha: \alpha \in \lambda\}$ in $CL(X)$ is said to be Wijsman convergent to some $A$ in $CL(X)$ if $d(x,A_\alpha)\rightarrow d(x,A)$ for every $x \in X$. The Wijsman topology on $CL(X)$ induced by $d$, denoted by $\tau_{w(d)}$, is the weakest topology such that for every $x \in X$, the distance functional

\bigskip
\centerline{$d(x,.): CL(X)  \to R^+$}
\bigskip
is continuous. It can be seen easily that the Wijsman topology on $CL(X)$ induced by $d$ has the family

\bigskip
\centerline{$\{U^-: U$ open in $X\} \cup \{\{A \in CL(X): d(x,A) > \epsilon\}: x \in X, \epsilon > 0\}$}

\bigskip
as a subbase [Be].

The above type of convergence was introduced by Wijsman in [Wi] for sequences of closed convex sets in Euclidean space $R^n$, when he considered optimum properties of the sequential probability ratio test.

It is  known [LL] that if $(X,d)$ is a metric space, then $(CL(X),\tau_{w(d)})$ is metrizable if and only if $(X,d)$ is separable.

One of the fundamental problems in the study of the Wijsman topology is to determine when it is normal. The problem was first mentioned by Di Maio and Meccariello in [DM], where it was asked whether the normality of the Wijsman topology is equivalent to its metrizability. Partial solutions to this problem were found by Hol\'a and Novotn\'y in [HN] and by Cao and Junnila in [CJ]. Hol\'a and Novotn\'y in [HN] showed that for a normed linear space $(X,d)$,  $(CL(X),\tau_{w(d)})$  is normal iff it is metrizable.  Cao and Junnila in [CJ] proved that a Wijsman hyperspace is hereditarily normal if and only if it is metrizable.

\bigskip
\bigskip
\bigskip

\section{Main result}

\bigskip

In [CJ]  the following Question 3.4 is posed.
\bigskip

{\bf Question 3.4. ([CJ])} Let $(X,d)$ be a non-separable metric space. Does $(CL(X),\tau_{w(d)})$ contain a copy of $(\omega_1 +1)^{\omega}$ or $(\omega_1 + 1)^{\omega_1}$?

\bigskip
\begin{theorem} Let $X$ be a $T_2$ topological space which contains an uncountable closed discrete set. $(CL(X),\tau_F)$ contains a copy of $(\omega_1 + 1)^{\omega}$ and $(\omega_1 + 1)^{\omega_1}$.

\end{theorem}

\begin{proof} We will prove the Theorem for the case of $(\omega_1 + 1)^{\omega_1}$, the other case is similar.  Let $D$ be a closed uncountable discrete set in $X$ such that $\mid D \mid = \aleph_1$.  We express $D$ as a  pairwise disjoint union of closed discrete sets  $\{D_i: i \in \omega_1\}$,

\bigskip
\centerline{$D = \bigcup_{i \in \omega_1} D_i$,}

\bigskip
such that $\mid D_i \mid = \aleph_1$ for every $i \in \omega_1$. Without loss of generality we can suppose that $X \ne D$
and let $l \in X \setminus D$. For each $x \in D$ fix an open neighbourhood $U(x)$ of $x$ such that $U(x) \cap (D \cup \{l\}) = \{x\}$.

For every $i \in \omega_1$ enumerate $D_i = \{x_{\alpha}^i: 0 < \alpha \le \omega_1\}$ and define a map

\bigskip
\centerline{$\varphi_i: \omega_1 + 1 \to 2^X$ }

\bigskip

as follows: $\varphi_i(0) = \emptyset$, $\varphi_i(\alpha) = \{x_{\eta}^i: \alpha \le \eta\}$ if $0 < \alpha \le \omega_1$.  Let $(\alpha_i)_{i \in \omega_1}$ be a point from $(\omega_1 + 1)^{\omega_1}$ its $i$th coordinate is $\alpha_i$. Now we will define a map

\bigskip
\centerline{$\phi: (\omega_1 + 1)^{\omega_1} \to CL(X)$}
\bigskip

as follows: $\phi((\alpha_i)_{i \in \omega_1}) = \{l\}$ if $\alpha_i = 0$ for every $i \in \omega_1$ and

\bigskip

\centerline{$\phi((\alpha_i)_{i \in \omega_1}) =  \{l\} \cup \bigcup_{i \in \omega_1}\varphi_i(\alpha_i)$}

\bigskip
if there is $i \in \omega_1$ such that $\alpha_i \ne 0$.

If we take distinct $A, B \in \phi((\omega_1 + 1)^{\omega_1})$, then there is $x \in D$ with $U(x)$ missing one of $A, B$ and hitting the other. So $U(x)^- \cap \phi((\omega_1 + 1)^{\omega_1})$ and $(X \setminus \{x\})^+ \cap \phi((\omega_1 + 1)^{\omega_1})$ are disjoint $\phi((\omega_1 + 1)^{\omega_1})$ neighbourhoods of $A, B$.

Consequently,  $(\phi((\omega_1 + 1)^{\omega_1}),\tau_F)$ is a Hausdorff space, so to show that $\phi: (\omega_1 + 1)^{\omega_1} \to  (\phi((\omega_1 + 1)^{\omega_1},\tau_F)$ is a homeomorphism it suffices to show that $\phi$ is continuous, since $(\omega_1 + 1)^{\omega_1}$ is compact and $\phi$ is one-to-one.

To prove the continuity of $\phi$ we first show that $\phi^{-1}(V^-)$ is open in $(\omega_1 + 1)^{\omega_1}$ for each open subset $V$ of $X$. So let $V \subseteq X$ be open. If $l \in V$ then $\phi^{-1}(V^-) = (\omega_1 + 1)^{\omega_1}$. Suppose $l \notin V$. Let $(\alpha_i)_{i \in \omega_1} \in \phi^{-1}(V^-)$. Thus $\phi((\alpha_i)_{i \in \omega_1}) \cap V \ne \emptyset$. There is $i \in \omega_1$ such that $\alpha_i \ne 0$ and  $\varphi_i(\alpha_i) \cap V \ne \emptyset$. It is easy to verify that $(0,\alpha_i]$ is an open neighbourhood of $\alpha_i$ such that $\varphi_i (\eta) \cap V \ne \emptyset$ for every $\eta \in (0,\alpha_i]$. Thus $\Pi_{j \in \omega_1} X_j$, where $X_i = (0,\alpha_i]$ and $X_j = \omega_1 + 1$ for $j \ne i$ is a neighbourhood of $(\alpha_i)_{i \in \omega_1}$ contained in $\phi^{-1}(V^-)$.

Now let $K$ be a compact set in $X$. We show that $\phi^{-1}((K^c)^+)$ is open in $(\omega_1 + 1)^{\omega_1}$. Let $(\alpha_i)_{i \in \omega_1} \in \phi^{-1}((K^c)^+)$. There is a finite set $J \subset \omega_1$ such that $K \cap D_j \ne \emptyset$ for every $j \in J$. For every $j \in J$, $\varphi_j(\alpha_j) \cap K = \emptyset$, thus there is an open neighbourhood $O(\alpha_j)$ of $\alpha_j$ such that $\varphi_j(\eta) \cap K = \emptyset$ for every $\eta \in O(\alpha_j)$. For every $i \in \omega_1$ put $X_i = O(\alpha_i)$ if $i \in J$ and $X_i = \omega_1 + 1$ otherwise. Then $\Pi_{i \in \omega_1} X_i$ is a neighbourhood of $(\alpha_i)_{i \in \omega_1}$ contained in $\phi^{-1}((K^c)^+)$.

\end{proof}

\bigskip
\bigskip
Notice that in the paper [HL] we proved that if  $X$ is a  $T_2$ topological space which contains an uncountable closed discrete set, then $\omega_1 \times (\omega_1 + 1)$ embeds into $(CL(X),\tau_F)$ as a closed set. It was shown in [HZ] that if  $X$ is a  $T_2$ topological space which contains an uncountable closed discrete set, then also Tychonoff plank embeds into $(CL(X),\tau_F)$ as a closed subspace.

\bigskip
\bigskip
\section{Concerning Question 3.4 in [CJ]}

\bigskip

In this part we will give a partial answer to the Question 3.4 in [CJ].
\bigskip

It is known [Be] that if a metric space $(X,d)$ has nice closed balls, then the Fell topology $\tau_F$ and the Wijsman topology $\tau_{w(d)}$ on $CL(X)$ coincide.

A metric space $(X,d)$ is said to have nice closed balls [Be] provided whenever $B$ is a closed ball in $X$ that is a proper subset of $X$, then $B$ is compact. The class of metric spaces that have nice closed balls includes those metric spaces in which closed and bounded sets are compact, as well as all $0-1$ metric spaces [Be].

\bigskip

\bigskip
\begin{theorem} Let $(X,d)$ be a non-separable metric space with nice closed balls. Then $(CL(X),\tau_{w(d)})$  contains a copy of $(\omega_1 + 1)^{\omega}$ and $(\omega_1 + 1)^{\omega_1}$.

\end{theorem}

The following theorem  gives a better  partial answer to the Question 3.4 in [CJ] than Theorem 3.1.

\begin{theorem} Let $(X,d)$ be a metric space such that every closed proper ball in $X$  is totally bounded. If $X$ is non-separable, then $(CL(X),\tau_{w(d)})$ contains a copy of $(\omega_1 + 1)^{\omega}$ and $(\omega_1 + 1)^{\omega_1}$.

\end{theorem}
\begin{proof} We will prove the Theorem for the case of $(\omega_1 + 1)^{\omega_1}$, the other case is similar.  Since  $(X,d)$ is non-separable there exist $\epsilon > 0$ and a set $D \subset X$ with $\mid D \mid = \aleph_1$ which is $\epsilon$-discrete, that is,  $d(x,y) \ge \epsilon$ for all distinct $x, y \in D$. We express $D$ as a  pairwise disjoint union of closed discrete sets  $\{D_i: i \in \omega_1\}$,

\bigskip
\centerline{$D = \bigcup_{i \in \omega_1} D_i$,}

\bigskip
such that $\mid D_i \mid = \aleph_1$ for every $i \in \omega_1$. Without loss of generality we can suppose that $X \ne D$
and let $l \in X \setminus D$.

For every $i \in \omega_1$ enumerate $D_i = \{x_{\alpha}^i: 0 < \alpha \le \omega_1\}$ as in the proof of Theorem 2.1 and we will proceed as in the proof of Theorem 2.1. We claim that $\phi: (\omega_1 + 1)^{\omega_1}  \to (CL(X),\tau_{w(d)})$ is embedding. Of course, it is sufficient to prove that $\phi$ is continuous.

It is sufficient to verify that if $d(x,\phi((\alpha_i^*)_{i \in \omega_1}))) > r$ for some $x \in X$ and $r > 0$, then
$d(x,\phi((\alpha_i)_{i \in \omega_1}))) > r$ for all $(\alpha_i)_{i \in \omega_1}$ from a neighbourhood of $(\alpha_i^*)_{i \in \omega_1}$. However, it is clear, since the closed proper ball with center $x$ and the radius $r$ can contain only finitely many points of the set $D$.

\end{proof}

A topological space $X$ is perfect if there are no isolated points in $X$.

\begin{theorem} Let $(X,d)$ be a non-separable perfect metric space. Then $(\omega_1 + 1)^{\omega}$ and $(\omega_1 + 1)^{\omega_1}$ embed into $(CL(X),\tau_{w(d)})$.
\end{theorem}

\begin{proof} We will prove the Theorem for the case of $(\omega_1 + 1)^{\omega_1}$, the other case is similar. Since $(X,d)$ is non-separable, there exist $\epsilon > 0$ and a set $D \subset X$ with $\mid D \mid = \aleph_1$ which is $\epsilon$-discrete, that is,  $d(x,y) \ge \epsilon$ for all distinct $x, y \in D$. Without loss of generality we can suppose that all points of $D$ are non-isolated. For every $x \in D$ there is a point $t_x \in X$ such that $d(x,t_x) < \epsilon/10$. For every $x \in D$ put $\eta(x) = d(x,t_x)$. Set

\bigskip

\centerline{$H = X \setminus \bigcup_{x \in D} S(x,\eta(x))$,}

\bigskip
where $S(x,\eta(x)) = \{s \in X: d(x,s) < \eta(x)\}$. We express $D$ as a  pairwise disjoint union of closed discrete sets  $\{D_i: i \in \omega_1\}$,

\bigskip
\centerline{$D = \bigcup_{i \in \omega_1} D_i$,}

\bigskip
such that $\mid D_i \mid = \aleph_1$ for every $i \in \omega_1$. For every $i \in \omega_1$ enumerate $D_i = \{x_{\alpha}^i: 0 <  \alpha \le \omega_1\}$ and define a map

\bigskip
\centerline{$\varphi_i: \omega_1 + 1 \to 2^X$ }
\bigskip

as follows: $\varphi_i(0) = \emptyset$, $\varphi_i(\alpha) = \{x_{\eta}^i: \alpha \le \eta\}$ if $0 < \alpha \le \omega_1$.  Let $(\alpha_i)_{i \in \omega_1}$ be a point from $(\omega_1 + 1)^{\omega_1}$ its $i$th coordinate is $\alpha_i$. Now we will define a map

\bigskip
\centerline{$\phi: (\omega_1 + 1)^{\omega_1} \to CL(X)$}
\bigskip

as follows: $\phi((\alpha_i)_{i \in \omega_1}) =  H$ if $\alpha_i = 0$ for every $i \in \omega_1$ and
\bigskip

\centerline{$\phi((\alpha_i)_{i \in \omega_1}) = H \cup  \bigcup_{i \in \omega_1}\varphi_i(\alpha_i)$ }
\bigskip

if there is $i \in \omega_1$ such that $\alpha_i \ne 0$.

To show that $\phi: (\omega_1 + 1)^{\omega_1} \to  (\phi((\omega_1 + 1)^{\omega_1},\tau_{w(d)})$ is a homeomorphism it suffices to show that $\phi$ is continuous, since $(\omega_1 + 1)^{\omega_1}$ is compact,  $(\phi((\omega_1 + 1)^{\omega_1},\tau_{w(d)})$ is a Hausdorff space and $\phi$ is one-to-one.

To prove the continuity of $\phi$ we first show that $\phi^{-1}(V^-)$ is open in $(\omega_1 + 1)^{\omega_1}$ for each open subset $V$ of $X$. So let $V \subseteq X$ be open. If $H \cap  V \ne \emptyset$ then $\phi^{-1}(V^-) = (\omega_1 + 1)^{\omega_1}$. Suppose $H \cap V = \emptyset$.

 Let $(\alpha_i)_{i \in \omega_1} \in \phi^{-1}(V^-)$. Thus $\phi((\alpha_i)_{i \in \omega_1}) \cap V \ne \emptyset$. There is $i \in \omega_1$ such that $\alpha_i \ne 0$ and  $\varphi_i(\alpha_i) \cap V \ne \emptyset$. It is easy to verify that $(0,\alpha_i]$ is an open neighbourhood  of $\alpha_i$ such that $\varphi_i(\eta) \cap V \ne \emptyset$ for every $\eta \in (0,\alpha_i]$. Thus $\Pi_{j \in \omega_1} X_j$, where $X_i = (0,\alpha_i]$ and $X_j = \omega_1 + 1$ for $j \ne i$ is a neighbourhood of $(\alpha_i)_{i \in \omega_1}$ contained in $\phi^{-1}(V^-)$.

It is sufficient to verify that if $d(x,\phi((\alpha_i^*)_{i \in \omega_1}))) > r > 0$ for some $x \in X$, then
$d(x,\phi((\alpha_i)_{i \in \omega_1}))) > r$ for all $(\alpha_i)_{i \in \omega_1}$ from a neighbourhood of $(\alpha_i^*)_{i \in \omega_1}$.  Suppose first that $(\alpha_i^*)_{i \in \omega_1}$ is such that $\alpha_i^* = 0$ for all $i  \in \omega_1$. Thus $\phi((\alpha_i^*)_{i \in \omega_1})= H$.

If $d(x,H) > r > 0$, then there is $z \in D$ such that $x \in S(z,\eta(z))$. There is $i \in \omega_1$ such that $z \in D_i$. Thus $\Pi_{j \in \omega_1} X_j$, where $X_i = \{0\} $ and $X_j = \omega_1 + 1$ for $j \ne i$ is a neighbourhood of $(\alpha_i^*)_{i \in \omega_1}$ such that $d(x,\phi((\alpha_i)_{i \in \omega_1}))) > r$ for all $(\alpha_i)_{i \in \omega_1}$ from the neighbourhood.

Suppose now that $d(x,\phi((\alpha_i^*)_{i \in \omega_1}))) > r > 0$ for some $x \in X$ and  $(\alpha_i^*)_{i \in \omega_1}$  such that there is $j \in \omega_1$ with $\alpha_j^* \ne 0$. There is $z \in D$ with $x \in S(z,\eta(z))$. If $z \in \phi((\alpha_i^*)_{i \in \omega_1})$, then $(\omega_1 + 1)^{\omega_1}$ is a neighbourhood of $(\alpha_i^*)_{i \in \omega_1}$ such that $d(x,\phi((\alpha_i)_{i \in \omega_1}))) > r$ for every $(\alpha_i)_{i \in \omega_1}$ from the neighbourhood $(\omega_1 + 1)^{\omega_1}$.

Suppose that $z \notin \phi((\alpha_i^*)_{i \in \omega_1})$. There are $i \in \omega_1, \beta \in \omega_1$ such that $z  = x_{\beta}^i$. If $\alpha_i^* = 0$, then $O = \Pi_{j \in \omega_1} X_j$, where $X_i = \{0\}$ and $X_j = \omega_1 + 1$ for $j \ne i$ is a neighbourhood of $(\alpha_i^*)_{i \in \omega_1}$ such that $d(x,\phi((\alpha_i)_{i \in \omega_1}))) > r$ for every $(\alpha_i)_{i \in \omega_1}$ from the neighbourhood $O$. If $\alpha_i^* \ne 0$, then $O = \Pi_{j \in \omega_1} X_j$, where $X_i = (\beta,\alpha_i^*]$ and $X_j = \omega_1 + 1$ for $j \ne i$ is a neighbourhood of $(\alpha_i^*)_{i \in \omega_1}$ such that $d(x,\phi((\alpha_i)_{i \in \omega_1}))) > r$ for every $(\alpha_i)_{i \in \omega_1}$ from the neighbourhood $O$.

\end{proof}

\bigskip

A topological space $X$ is perfect if there are no isolated points in $X$.

\begin{theorem} Let $(X,d)$ be a non-separable perfect metric space. Then $(\omega_1 + 1)^{\omega}$ and $(\omega_1 + 1)^{\omega_1}$ embed into $(CL(X),\tau_{w(d)})$.
\end{theorem}

\end{document}